\documentclass[11pt]{amsart}
\usepackage{amssymb}
\usepackage{amsmath}
\usepackage{amsthm}
\usepackage{verbatim}
\usepackage[all]{xy}
\usepackage{url}
\usepackage{mathtools}
\usepackage{dsfont}

\DeclareSymbolFont{rsfs}{U}{rsfs}{m}{n}
\DeclareSymbolFontAlphabet{\mathrsfs}{rsfs}

\DeclareFontEncoding{OT2}{}{} 
\DeclareTextFontCommand{\textcyr}{\fontencoding{OT2}
    \fontfamily{wncyr}\fontseries{m}\fontshape{n}\selectfont}

\newcommand{\Ch}{\textcyr{Ch}}

\theoremstyle{plain}
\newtheorem{theorem}{Theorem}[section]

\newtheorem{proposition}[theorem]{Proposition}
\newtheorem{lemma}[theorem]{Lemma}
\newtheorem{corollary}[theorem]{Corollary}
\newtheorem{maintheorem}[theorem]{Main Theorem}

\theoremstyle{definition}

\newtheorem{remark}[theorem]{Remark}

\newtheorem{subsecc}[theorem]{}
\numberwithin{equation}{theorem}

\def\C{{\mathds C}}
\newcommand{\R}{{\mathds R}}
\newcommand{\Q}{{\mathds Q}}

\newcommand{\Z}{{\mathds Z}}

\newcommand{\G}{{\mathbb G}}
\renewcommand{\H}{{\mathds H}}

\newcommand{\ov}{\overline}
\newcommand{\wh}{\widehat}

\newcommand{\lra}{\longrightarrow}

\DeclareMathOperator{\Aut}{Aut}

\DeclareMathOperator{\coker}{coker}

\DeclareMathOperator{\Gal}{Gal}

\DeclareMathOperator{\im}{im}

\newcommand{\ad}{{\rm ad}}
\newcommand{\ab}{{\rm ab}}

\newcommand{\ssc}{{\rm sc}}

\newcommand{\into}{\hookrightarrow}
\newcommand{\onto}{\twoheadrightarrow}

\newcommand{\labelt}[1]{\xrightarrow{\makebox[1.em]{\scriptsize ${#1}$}}}
\newcommand{\labelto}[1]{\xrightarrow{\makebox[1.5em]{\scriptsize ${#1}$}}}
\newcommand{\labeltoo}[1]{\xrightarrow{\makebox[2em]{\scriptsize ${#1}$}}}
\newcommand{\labeltooo}[1]{\xrightarrow{\makebox[2.7em]{\scriptsize ${#1}$}}}

\newcommand{\longisoto}{{\ \labelt{\raisebox{-2.ex}{$\sim$}}\ }}

\newcommand{\isoto}{\longisoto}

\newcommand{\hs}{\kern 0.8pt}
\newcommand{\hssh}{\kern 1.2pt}
\newcommand{\hshs}{\kern 1.6pt}
\newcommand{\hssss}{\kern 2.0pt}

\newcommand{\hm}{\kern -0.8pt}
\newcommand{\hmm}{\kern -1.2pt}

\newcommand{\emm}{\bfseries}

\newcommand{\X}{{{\sf X}}}

\newcommand{\Fbar}{{\ov F}}
\newcommand{\loc}{{\rm loc}}

\newcommand{\Vm}{{\mathrsfs V}}

\renewcommand{\G}{\Gamma}

\newcommand{\tors}{{\rm Tors}}

\def\g{\gamma }

\newcommand{\Gt}{{\G\hm,\hs\tors}}

\def\Gtf{{\G\!,\hs\tf}}

\newcommand{\Tors}{{\rm Tors}}
\newcommand{\tf}{{\rm\hs t.f.}}
\newcommand{\upgam}{{{}^\gamma\hskip-1.0pt}}

\def\Gw{{\G\!_w}}
\def\Gwt{{\Gw,\hs\Tors}}

\def\SC{{S^\complement}}

\newcommand{\Tor}{\mathrm{Tor}}
\newcommand{\Lam}{{\Lambda}}

\newcommand{\rsa}{\rightsquigarrow}

\newcommand{\Fs}{{F_s}}

\begin{document}

\title[Galois cohomology]
{Criterion for surjectivity of localization\\
in Galois cohomology of a reductive group\\
over a number field}

\author{Mikhail Borovoi}

\address{Borovoi: Raymond and Beverly Sackler School of Mathematical Sciences,
Tel Aviv University, 6997801 Tel Aviv, Israel}
\email{borovoi@tauex.tau.ac.il}

\address{%
Rosengarten: Einstein Institute of Mathematics,
The Hebrew University of Jerusalem,
Edmond J. Safra Campus, 91904 Jerusalem, Israel}
\email{zev.rosengarten@mail.huji.ac.il}



\thanks{The author was partially supported
by the Israel Science Foundation (grant 1030/22).
}

\begin{abstract}
Let $G$ be a connected reductive group over a number field $F$,
and let $S$ be a set (finite or infinite) of places of $F$.
We give a necessary and sufficient condition
for the surjectivity of the localization map from $H^1(F,G)$
to the ``direct sum'' of the sets $H^1(F_v,G)$ where $v$ runs over $S$.
In the appendices, we give a new construction of the abelian Galois cohomology
of a reductive group over a field of arbitrary characteristic.
\\

\noindent{\sc R\'esum\'e.} Soit $G$ un groupe réductif connexe sur un corps de nombres $F$,
et soit $S$ un ensemble (fini ou infini) de places de $F$.
On donne une condition nécessaire et suffisante
pour la surjectivité de l'application de localisation de $H^1(F,G)$
vers la ``somme directe'' des ensembles $H^1(F_v,G)$, où $v$ parcourt $S$.
Dans les appendices on donne une nouvelle construction
de la cohomologie galoisienne ab\'elienne
d'un groupe r\'eductif sur un corps de caract\'eristique quelconque.
\end{abstract}

\dedicatory{With an appendix by Zev Rosengarten}

\maketitle



\section{Introduction}
\label{s:intro}

\begin{subsecc}
Let $G$ be a (connected) reductive group over a number field $F$
(we follow the convention of SGA3, where reductive groups
are assumed to be connected).
Let $\Fbar$ be a fixed algebraic closure of $F$.
We denote by $\Vm(F)$ the set of places of $F$.
For $v\in\Vm(F)$, we denote by $F_v$ the completion of $F$ at $v$.
We refer to Serre's book \cite{Serre}
for the definition of the first Galois cohomology set $H^1(F,G)$.

In general, $H^1(F,G)$ is just a pointed set and has no natural groups structure.
Let $H^1_\ab(F,G)$ denote the {\em abelian Galois cohomology group} of $G$
introduced in \cite[Section 2]{Borovoi-Memoir};
see also Labesse \cite[Section 1.3]{Labesse}.
This is an abelian group depending functorially on $G$ and $F$.
There is a canonical {\em abelianization map}
\[\ab\colon H^1(F,G)\to H^1_\ab(F,G).\]
We give a new, better construction of $H^1_\ab(F,G)$ in Appendix \ref{s:ab}.

Let $S\subseteq \Vm(F)$ be a subset (finite or infinite).
We consider the  localization map
\begin{align}\label{e:loc-S-ab-prod}
 H^1_\ab(F,G)\to\prod_{v\in S} H^1_\ab(F_v,G).
\end{align}
In fact this map takes values in the subgroup $\bigoplus_{v\in S} H^1_\ab(F_v,G)\subseteq \prod_{v\in S} H^1_\ab(F_v,G)$;
see \cite[Corollary 4.6]{Borovoi-Memoir}.
Thus we obtain a localization map
\begin{equation}\label{e:loc-S-ab}
\loc_S^\ab\colon H^1_\ab(F,G)\to\bigoplus_{v\in S} H^1_\ab(F_v,G).
\end{equation}
Similarly, consider the localization map
\begin{equation*}
 H^1(F,G)\to\prod_{v\in S} H^1(F_v,G).
\end{equation*}
In fact it takes values in the subset $\bigoplus_{v\in S} H^1(F_v,G)$ consisting
of the families $(\xi_v)_{v\in S}$
with $\xi_v\in H^1(F_v,G)$ and such that $\xi_v=1$ for all $v$
except maybe finitely many of them.
This well-known fact follows, for instance,
from the corresponding assertion for \eqref{e:loc-S-ab-prod}
together with \cite[Theorem 5.11 and Corollary 5.4.1]{Borovoi-Memoir}.
Thus we obtain a localization map
\begin{equation}\label{e:loc-S}
\loc_S\colon H^1(F,G)\to\bigoplus_{v\in S} H^1(F_v,G).
\end{equation}
We wish to find conditions under which the localization maps \eqref{e:loc-S-ab} and \eqref{e:loc-S} are surjective.
\end{subsecc}

\begin{subsecc}
We denote by $M=\pi_1(G)$ the {\em algebraic fundamental group of $G$}
(also known as the Borovoi fundamental group of $G$)
introduced in \cite[Section 1]{Borovoi-Memoir},
and also introduced by Merkurjev \cite[Section 10.1]{Merkurjev}
and Colliot-Th\'el\`ene \cite[Proposition-Definition 6.1]{CT-flasque}.
See Subsection \ref{ss:pi1} for our definition of $\pi_1 (G)$.
This is a finitely generated abelian group,
on which the absolute Galois group $\Gal(\Fbar/F)$ naturally acts.
Let $E/F$ be a finite Galois extension in $\Fbar$ such that
$\Gal(\Fbar/E)$ acts on $M$ trivially and that $E$ has no real places.
Then the Galois group $\G\coloneqq\Gal(E/F)$ naturally acts on $M$
and on the set of places $\Vm(E)$ of the field $E$.
\end{subsecc}

\begin{subsecc}
We denote by $\Ch^1_S(F,G)$ the cokernel of the homomorphism \eqref{e:loc-S-ab}, that is,
\[ \Ch^1_S(F,G)=\coker\bigg[\loc_S^\ab \colon
     H^1_\ab(F,G)\to\bigoplus_{v\in S} H^1_\ab(F_v,G)\bigg].\]
After explaining our notation in Section \ref{s:notation},
we compute  in Section \ref{s:main} the finite abelian group $\Ch^1_S(F,G)$
in terms of the action of $\G$ on $M$ and on $\Vm(E)=\Vm_f(E)\cup\Vm_\C(E)$;
see Corollary \ref{c:main-ab}.
See Subsection \ref{ss:mt1} for the notations $\Vm_f$ and $\Vm_\C$.

Concerning the map $\loc_S$ of \eqref{e:loc-S}, in Section \ref{s:main} we
compute the image of this map; see Main Theorem \ref{t:main}.
Using this result, we give a {\em criterion} (necessary and sufficient condition)
for the map $\loc_S$ to be surjective; see Corollary \ref{c:main}.
This is also a criterion for the vanishing of $\Ch^1_S(F,G)$.
Again, our criterion is given in terms of the action
of $\G$ on $M$ and on $\Vm(E)=\Vm_f(E)\cup\Vm_\C(E)$.
Using this criterion,  we give a simple proof
of the result of Borel and Harder \cite[Theorem 1.7]{Borel-Harder}
(see also Prasad and Rapinchuk \cite[Proposition 1]{Prasad-Rapinchuk}\hs)
on the surjectivity of the map $\loc_S$ when $G$ is semisimple
and  there exists a finite place $v_0$ of $F$ outside $S$;
see Proposition \ref{p:semisimple} below.

Let $\Gamma$ be a finite group.
In Section \ref{s:exact}, we construct
an exact sequence arising
from a short exact sequence of $\Gamma$-modules.
In Section \ref{s:PR}, using this exact sequence
and Main Theorem \ref{t:main},
we generalize a result of Prasad and Rapinchuk
giving a sufficient condition for the surjectivity
of the localization map $\loc_S$ when $G$ is reductive,
in terms of the radical (largest central torus) of $G$;
see Theorem \ref{t:PR-generalized}.
As a particular case, we obtain the following corollary.
\end{subsecc}

\begin{corollary}[of Theorem \ref{t:PR-generalized}]
\label{c:Prasad-Rapinchuk}
Let $G$ be a reductive group over a number field $F$,
and let $C$ denote the radical of $G$
(the identity component of the center of $G$).
Let $S\subset \Vm(F)$ be a set of places of $F$.
Assume that the $F$-torus $C$ splits
over a finite Galois extension of $F$ of prime degree $p$
and that there exists a finite place  $v_0$
in the complement $\SC\coloneqq \Vm(F)\smallsetminus S$ of $S$
such that $C$ does not split over $F_{v_0}$\hs.
Then the localization map $\loc_S$ of \eqref{e:loc-S} is surjective.
\end{corollary}

For $p=2$ this assertion was earlier proved
by Prasad and Rapinchuk \cite[Proposition 2(b)]{Prasad-Rapinchuk}.

\begin{subsecc}
Let $G$ be a reductive group over a field $F$ of  characteristic 0.
In \cite{Borovoi-Memoir}, the author defined
the abelian group $H^1_\ab(F,G)$ {\em as a set}
in a canonical way as the Galois hypercohomology of a certain crossed module.
However, the definition of the structure of abelian group
on $H^1_\ab(F,G)$ in \cite{Borovoi-Memoir} was complicated.
In Appendix \ref{s:ab}, we define $H^1_\ab(F,G)$  (in arbitrary characteristic)
following the letter of Breen to the author \cite{Breen}
and the article by Noohi \cite{Noohi} (written at the author's request),
as the Galois hypercohomology $\H^1(F,G_\ab)$
of a  certain {\em stable} crossed module, that is,
a crossed module endowed with a symmetric braiding.
The structure of abelian group comes from the symmetric braiding.
Note that our specific crossed module and specific symmetric braiding
were constructed by Deligne \cite{Deligne}.

In Appendix \ref{s:Zev},
Zev Rosengarten shows that certain equivalences of crossed modules
of algebraic groups over a field $F$ of arbitrary characteristic
induce equivalences on $\Fs$-points where $\Fs$ is a separable closure of $F$.
This permits us to use in Appendix \ref{s:ab}
the {\em Galois} hypercohomology of these crossed modules
rather than fppf hypercohomology.
\end{subsecc}

\section{Notation}
\label{s:notation}

\begin{subsecc}
Let $A$ be an abelian group. We denote by $A_\Tors$ the torsion subgroup of $A$.
We set $A_\tf=A/A_\Tors$\hs, which is a torsion-free group.
\end{subsecc}

\begin{subsecc}
Let $\G$ be a finite group, and let $B$ be a $\G$-module.
We denote by $B_\G$ the group of coinvariants of $\G$ in $B$, that is,
\[ B_\G=B\hs \big/\bigg\{\sum_{\gamma\in \G}(\,{}^{\gamma^{-1}}\!b_\gamma-b_\gamma\,)
\ \big|\ b_\gamma\in B\bigg\}.\]
We write $B_\Gt\coloneqq (B_\G)_\Tors$ (which is the torsion subgroup of $B_\G$),
\ $B_\Gtf=B_\G/ B_\Gt$\ (which is a torsion-free group).
\end{subsecc}

\begin{subsecc}\label{ss:pi1}
Let $G$ be a reductive group over a field $F$.
Let $[G,G]$ denote the commutator subgroup of $G$, which is semisimple.
Let  $G^\ssc$ denote the universal cover of $[G,G]$,
which is simply connected;
see \cite[Proposition (2.24)(ii)]{Borel-Tits-C}
or \cite[Corollary A.4.11]{CGP}.
Following Deligne \cite[Section 0.2]{Deligne}, we consider the composite homomorphism
\[\rho\colon G^\ssc\onto [G,G]\into G,\]
which in general is neither injective nor surjective.

For a maximal torus  $T\subseteq G$, we write $T^\ssc=\rho^{-1}(T)\subseteq G^\ssc$
and consider the natural homomorphism
\[\rho\colon T^\ssc\to T.\]
We consider the  algebraic fundamental group $M=\pi_1 (G)$
of $G$ defined by
\[\pi_1 (G)=\X_*(T)/\rho_*\X_*(T^\ssc)\]
where $\X_*$ denotes the cocharacter group.
The Galois group $\Gal(\Fs/F)$ naturally acts on $M$,
and the $\Gal(\Fs/F)$-module $M$
is well defined (does not depend on the choice of $T$
up to a transitive system of isomorphisms);
see \cite[Lemma 1.2]{Borovoi-Memoir}.
\end{subsecc}

\begin{subsecc}\label{ss:mt1}
From now on  (except for the appendices), $F$ is  a number field.
We denote by  $\Vm(F)$, $\Vm_f(F)$, $\Vm_\infty(F)$, $\Vm_\R(F)$, and $\Vm_\C(F)$
the sets of all places of $F$,  of finite places, of infinite places, of real places,
and of complex places, respectively.

Let $E/F$ be a finite Galois extension of number fields with Galois group $\G=\Gal(E/F)$;
then $\G$ acts on $\Vm(E)$.
If $w\in\Vm(E)$, we write $\Gw$ for the stabilizer of $w$ in $\G$;
then $\Gw\cong \Gal(E_w/F_v)$ where $v\in\Vm(F)$ is the restriction of $w$ to $F$.
\end{subsecc}

\section{Main theorem}
\label{s:main}

In this section we state and prove Main Theorem \ref{t:main}
computing the images of the localization maps \eqref{e:loc-S-ab} and \eqref{e:loc-S}.
We deduce Corollary \ref{c:main-ab} computing the group $\Ch^1_S(F,G)$,
and Corollary \ref{c:main} giving a necessary and sufficient condition
for the surjectivity of the localization map \eqref{e:loc-S}.

\begin{subsecc}\label{ss:local-results}
Let $G$ be a  reductive group over a number field $F$,
and let $v\in\Vm_f(F)$ be a finite place of $F$.
In \cite{Borovoi-Memoir} we computed $H^1_\ab(F_v,G)$.
Write $M=\pi (G)$.
Let $E/F$ be a finite Galois extension  in $\Fbar$
such that $\Gal(\Fbar/E)$ acts on $M$ trivially and that $E$ has no real places.
Write $\G=\Gal(E/F)$.
\end{subsecc}

\begin{theorem} [{\cite[Proposition 4.1(i) and Corollary 5.4.1]{Borovoi-Memoir}}]
\label{t:4.1-Memoir}
With the notation and assumptions of Subsection \ref{ss:local-results},
for any {\emm finite} place $v$ of $F$
there is a canonical isomorphism of abelian groups
\[ \alpha_v^\ab\colon H^1_\ab(F_v,G)\isoto M_\Gwt\]
where $w$ is a place of $E$ over $v$, and a canonical bijection
\[\ab_v\colon H^1(F_v,G)\to H^1_\ab(F_v,G).\]
\end{theorem}

\begin{subsecc}
Let $v$ be a finite place of $F$.
We have a surjective (even bijective) map
\[\alpha_v\colon H^1(F_v,G)\labelto{\ab_v} H^1_\ab(F_v,G)\labelto{\alpha_v^\ab} M_\Gwt\hs.\]
We consider two composite maps with the same image
\begin{align*}
\lambda_v^\ab\colon\,  &H^1_\ab(F_v,G)\labelto{\alpha_v^\ab} M_\Gwt\labelto{\omega_v} M_\Gt\hs,\\
\lambda_v\colon\,      &H^1(F_v,G)\labelto{\alpha_v} M_\Gwt\labelto{\omega_v} M_\Gt\hs,
\end{align*}
where\, $\omega_v\colon  M_\Gwt\to M_\Gt$\, is
the homomorphism induced by the inclusion $\Gw\into\G$.
Since the maps $\alpha_v^\ab$ and $\alpha_v$ are surjective (even bijective), and $\omega_v$ is a homomorphism,
we see that the set\, $\im\lambda_v^\ab=\im\lambda_v$\, is a subgroup of $M_\Gt$\hs,
namely,\, $\im\lambda_v^\ab=\im\lambda_v=\im\hs\omega_v$\hs.

Let $v\in\Vm_\C(F)$ be a complex place. We have zero maps
\[\lambda_v^\ab\colon H^1_\ab(F_v,G)=\{1\}\to \{0\}\subseteq M_\Gt\hs,\qquad
     \lambda_v\colon H^1(F_v,G)=\{1\}\to \{0\}\subseteq M_\Gt\hs.\]
Clearly, the set\, $\im\lambda^\ab_v=\im\lambda_v$\,
is a subgroup of $M_\Gt$\hs, namely, the subgroup  $\{0\}$.
\end{subsecc}

\begin{subsecc}\label{ss:real}
Let  $v\in\Vm_\R(F)$ be a real place; then $\Gw$ is a group of order 2,
$\Gw=\{1,\gamma\}$ where $\gamma=\gamma_w$ induces
 the nontrivial automorphism of $E_w$ over $F_v$\hs.
We consider the Tate cohomology group
\[\wh H^{-1}(\Gw,M)=\{m\in M\mid \upgam m=-m\}\hs/\hss\{m'-\upgam m'\mid m'\in M\}.\]
We see immediately that the abelian group $\wh H^{-1}(\Gw,M)$
naturally embeds into $M_{\Gw}$\hs.
If $m\in M$ is a $(-1)$-cocycle, that is, $\upgam m=-m$,
then $2m=m+m=m-\upgam m$, whence $2\cdot\wh H^{-1}(\Gw,M)=0$.
We conclude that  $\wh H^{-1}(\Gw,M)$ naturally embeds into $M_{\Gw,\Tors}$\hs.

There is a canonical surjective map of Kottwitz \cite[Theorem 1.2]{Kottwitz-86}
(see also \cite[Theorem 5.4]{Borovoi-Memoir})
\[\ab_v\colon H^1(F_v, G)\onto H^1_\ab(F_v,G),\]
a canonical isomorphism of \cite[Proposition 8.21]{BT-arXiv}
\[ H^1_\ab(F_v,G)\isoto \wh H^{-1}(\Gw, M),\]
and a canonical embedding
\[\wh H^{-1}(\Gw,M)\into M_\Gwt\hs.\]
Thus we obtain  composite maps
\begin{align*}
\alpha_v^\ab\colon &H^1_\ab(F_v,G)\isoto\wh H^{-1}(\Gw, M)\into M_\Gwt\hs,\\
\alpha_v\colon  H^1(F_v, G)\onto &H^1_\ab(F_v,G)\isoto\wh H^{-1}(\Gw, M)\into M_\Gwt\hs,
\end{align*}
with the same image\, $\im \alpha_v^\ab=\im\alpha_v$\hs,\,
which is a subgroup of $M_\Gwt$\hs.
Consider the composite maps with the same image
\begin{align*}
\lambda_v^\ab\colon\, &H^1_\ab(F_v, G)\labelto{\alpha_v^\ab}  M_\Gwt\labelto{\omega_v} M_\Gt\hs,\\
\lambda_v\colon\, &H^1(F_v, G)\labelto{\alpha_v}  M_\Gwt\labelto{\omega_v} M_\Gt\hs.
\end{align*}
Since  the set\, $\im\alpha_v^\ab=\im\alpha_v$\,
is a subgroup of $M_\Gwt$\hs, and $\omega_v$ is a homomorphism,
we conclude that the set\, $\im\lambda_v^\ab=\im\lambda_v$\, is a subgroup of $M_\Gt$\hs,
namely,\,  $\im\lambda_v^\ab=\im\lambda_v=\omega_v\big(\wh H^{-1}(\Gw,M)\hs\big)$.
\end{subsecc}

\begin{lemma}\label{ss:loc-sum}
Let $S\subseteq \Vm(F)$ be any subset, finite or infinite.
Consider the  summation maps
\begin{align*}
\Sigma_S^\ab\colon &\bigoplus_{v\in S} H^1_\ab(F_v,G)\,\lra\, M_\Gt,\quad
   \   \xi_S^\ab= \big(\xi_v^\ab\big)_{v\in S}\,\longmapsto\, \sum_{v\in S}\lambda_v^\ab(\xi_v^\ab),\\
\Sigma_S\colon &\bigoplus_{v\in S} H^1(F_v,G)\,\lra\, M_\Gt,\quad
   \ \ \xi_S= \big(\xi_v\big)_{v\in S}\,\longmapsto\, \sum_{v\in S}\lambda_v(\xi_v).
\end{align*}
Then the sets \,$\im\Sigma_S^\ab$ and \,$\im \Sigma_S$ are  subgroups of $M_\Gt$\hs, and they are equal.
\end{lemma}

\begin{proof}
Indeed, we have
\[\im\Sigma^\ab_S=\im\Sigma_S=\big\langle \im\lambda_v\big\rangle_{v\in S}
    \quad\ \text{where}\ \,\im\lambda_v=
\begin{cases}
\im\omega_v                             &\text{if $v\in\Vm_f(F)$,}\\
\omega_v\big(\wh H^{-1}(\Gw,M)\hs\big) &\text{if $v\in\Vm_\R(F)$,}\\
0                                       &\text{if $v\in\Vm_\C(F)$.}
\end{cases}
\]
Here we write\, $\big\langle \im\lambda_v\big\rangle_{v\in S}$\,
for the subgroup of $M_\Gt$ generated by the subgroups\, $\im\lambda_v$\, for $v\in S$.
\end{proof}

\begin{theorem}
\label{t:Kottwitz}
The following sequences are  exact:
\begin{align}
&H^1_\ab(F,G)\labeltooo{\loc_{\Vm}^\ab}\bigoplus_{v\in \Vm} H^1_\ab(F_v,G)
       \labeltooo{\Sigma_{\Vm}^\ab} M_\Gt\hs,\label{a:1}
\\
&H^1(F,G)\labeltooo{\loc_{\Vm}}\bigoplus_{v\in \Vm} H^1(F_v,G)
       \labeltooo{\Sigma_{\Vm}} M_\Gt\hs,\label{a:2}
\end{align}
where for brevity we write $\Vm$ for $\Vm(F)$.
\end{theorem}
Here \eqref{a:1} is an exact sequence of abelian groups,
and \eqref{a:2} is an exact sequence of pointed sets.
\begin{proof}
In view of \cite[Proposition 4.8]{Borovoi-Memoir},  exact sequence \eqref{a:1}
is actually a part of the exact sequence \cite[(4.3.1)]{Borovoi-Memoir}.
For \eqref{a:2}, see \cite[Proposition 2.6]{Kottwitz-86} or \cite[Theorem 5.15]{Borovoi-Memoir}.
\end{proof}

\begin{maintheorem}\label{t:main}
Let $G$ be a reductive group over a number field $F$.
Let $S\subseteq \Vm\coloneqq \Vm(F)$ be a subset.
Write $\SC=\Vm\smallsetminus S$, the complement of $S$ in $\Vm$. Then:
\begin{align}
\im\hs \loc_S^\ab &=\bigg\{\xi_S^\ab\in \bigoplus_{v\in S} H^1_\ab(F_v,G)
       \ \big|\ \Sigma_S^\ab(\xi_S^\ab)\in
       \im\Sigma_S^\ab\hs\cap\hs\im\Sigma_\SC^\ab\bigg\},\label{a:main-1}\\
\im\hs \loc_S &=\bigg\{\xi_S\in \bigoplus_{v\in S} H^1(F_v,G)
       \ \big|\ \Sigma_S(\xi_S)\in\im\Sigma_S\hs\cap\hs\im\Sigma_\SC\bigg\}\label{a:main-2}.
\end{align}
\end{maintheorem}

\begin{proof}
By Lemma \ref{ss:loc-sum}, the sets $\im\Sigma_\SC^\ab$ and $\im\Sigma_\SC$
are (equal) subgroups of $M_\Gt$\hs, and therefore
it suffices to prove \eqref{a:main-1} with $\big(\!-\im\Sigma_\SC^\ab\,\big)$
instead of \,$\im\Sigma_\SC^\ab$\,, and to prove \eqref{a:main-2} with $\big(\!-\im\Sigma_\SC\hs\big)$
instead of \,$\im\Sigma_\SC$\hs.
Now the corresponding assertions follow easily
from the exactness of \eqref{a:1} and \eqref{a:2}, respectively.

For the reader's convenience, we provide an easy proof of \eqref{a:main-2}
with $\big(\!-\im\Sigma_\SC\hs\big)$ instead of \,$\im\Sigma_\SC$\hs.
Let
$$\xi_S=\big(\xi_v\big)_{v\in S}\,\in\,\im\loc_S\,\subseteq\,  \bigoplus_{v\in S} H^1(F_v,G),$$
that is,  $\xi_S=\loc_S(\xi)$ for some $\xi\in H^1(F,G)$.
Write $\eta_\SC=\big(\eta_v\big)_{v\in\SC}=\loc_\SC(\xi)$.
Since the sequence \eqref{a:2} is exact, we have $(\Sigma_\Vm\circ\loc_\Vm)(\xi)=0$, whence
\[ \Sigma_S(\xi_S)+\Sigma_\SC(\eta_\SC)=0\quad\ \text{and}\quad\ \Sigma_S(\xi_S)=-\Sigma_\SC(\eta_\SC).\]
We conclude that $\Sigma_S(\xi_S)\in \im\Sigma_S\cap\big(\!-\im \Sigma_\SC\big)$, as required.

Conversely, let an element  $\xi_S=\big(\xi_v\big)_{v\in S}\in \bigoplus_{v\in S} H^1(F_v,G)$
be such that $$\Sigma_S(\xi_S)\in \im\Sigma_S\hs\cap\big(\!-\im\Sigma_\SC\big).$$
Write $a=\Sigma_S(\xi_S)$. Then
$-a\in \im\Sigma_\SC$\hs, that is,
\[-a=\Sigma_\SC(\eta_\SC)\quad\text{for some}\quad \eta_\SC=
    \big(\eta_v\big)_{v\in\SC}\hs\in\hs \bigoplus_{v\in \SC} H^1(F_v,G).\]
Define
\[\zeta_{\Vm}=\big(\zeta_v\big)_{v\in \Vm}\,\in\, \bigoplus_{v\in \Vm} H^1(F_v,G), \quad\ \zeta_v=
\begin{cases}
\xi_v &\text{if $v\in S$,}\\
\eta_v &\text{if $v\in \SC$.}
\end{cases}
\]
Then
\[\Sigma_{\Vm}(\zeta_{\Vm})=a+(-a)=0.\]
Since the sequence \eqref{a:2} is exact, we have
$\zeta_\Vm=\loc_\Vm(\zeta)$ for some $\zeta\in H^1(F,G)$.
Then $\loc_S(\zeta)=\xi_S$, whence $\xi_S\in\im\loc_S$, as required.
\end{proof}

\begin{corollary}\label{c:main-ab}
The  homomorphism
\[\chi_S^\ab\colon \bigoplus_{v\in\Vm} H^1_\ab(F_v,G)\labelto{\Sigma_S^\ab} \im\Sigma_S^\ab
     \lra  \im \Sigma_S^\ab/\big(\im\Sigma_S^\ab\cap\im\Sigma_\SC^\ab\big)\]
induces a canonical isomorphism
\[\Ch^1_S(F,G)\isoto \im \Sigma_S^\ab/\big(\im\Sigma_S^\ab\cap\im\Sigma_\SC^\ab\big).\]
\end{corollary}

\begin{proof}
The homomorphism $\chi_S^\ab$ is clearly surjective,
and by Theorem \ref{t:main} its kernel is the image  $\im\loc_S^\ab$
of the localization homomorphism $\loc_S^\ab$ of \eqref{e:loc-S-ab}.
The corollary follows.
\end{proof}

\begin{corollary}\label{c:main}
The localization map $\loc_S$ of \eqref{e:loc-S} is surjective if and only if
\begin{equation}\label{e:main-formula}
\im \Sigma_S\,\subseteq\, \im\Sigma_\SC\hs.
\end{equation}
\end{corollary}

\begin{proof}
Consider the map
\[\chi_S\colon \bigoplus_{v\in S} H^1(F_v,G)\labelto{\Sigma_S}
      \im\Sigma_S \lra  \im \Sigma_S/\big(\im\Sigma_S\cap\im\Sigma_\SC\big).\]
By Lemma \ref{ss:loc-sum} the sets  \ $\im \Sigma_S= \im \Sigma_S^\ab$ \ and
\ $\im\Sigma_S\cap\im\Sigma_\SC=\im\Sigma_S^\ab\cap\im\Sigma_\SC^\ab$ \ are abelian groups.
The morphism of pointed sets $\chi_S$ is clearly surjective,
and by Theorem  \ref{t:main} its kernel is $\im\loc_S$.
We see that the following assertions are equivalent:
\begin{enumerate}
\item[(a)] the map $\loc_S$ is surjective, that is, \,$\im\loc_S=\bigoplus_{v\in S} H^1(F_v,G)$;
\item[(b)] $\ker \chi_S=\bigoplus_{v\in S} H^1(F_v,G)$;
\item[(c)] $\#(\im\chi_S) = 1$;
\item[(d)] $\im\Sigma_S\cap\im\Sigma_\SC=\im\Sigma_S$\hs;
\item[(e)] $\im\Sigma_S\subseteq\im \Sigma_\SC$\hs.
\end{enumerate}
This completes the proof.
\end{proof}

\begin{remark}
Since by Lemma \ref{ss:loc-sum} we have  $\im\Sigma_S^\ab=\im\Sigma_S$\,
and\,  $\im\Sigma_\SC^\ab=\im\Sigma_\SC$,\,
we see from (d) in the proof above and from Corollary \ref{c:main-ab}
that the localization map $\loc_S$ of \eqref{e:loc-S} is surjective
if and only if $\Ch^1_S(F,G)=\{1\}$.
\end{remark}

\begin{corollary}
Let $v_0\in\Vm(F)$, $S=\Vm(F)\smallsetminus\{v_0\}$.
Then the localization map $\loc_S$ of \eqref{e:loc-S} is surjective if and only if
\begin{equation}\label{e:lambda-lambda0}
\im\lambda_v\subseteq\im\lambda_{v_0}\quad\ \text{for all}\ \, v\in\Vm(F).
\end{equation}
\end{corollary}

\begin{proof}
Indeed, in our case condition \eqref{e:lambda-lambda0} is equivalent to \eqref{e:main-formula},
and we conclude by Corollary \ref{c:main}.
\end{proof}

\begin{corollary}
For a subset $S\subset \Vm(F)$,
let $v_0\in\SC$, and assume that
\begin{equation}\label{e:lambda-lambda0-bis}
\im\lambda_v\subseteq \im\lambda_{v_0}\quad\ \text{for all}\ \, v\in S.
\end{equation}
Then the localization map $\loc_S$ of \eqref{e:loc-S} is surjective.
\end{corollary}

\begin{proof}
Indeed, \eqref{e:lambda-lambda0-bis} implies \eqref{e:main-formula},
and we conclude by Corollary \ref{c:main}.
\end{proof}

\begin{corollary}\label{c:v0-sur}
Let $v_0\in \SC$, and assume that the map
$\lambda_{v_0}\colon H^1(F_{v_0},G)\to M_\Gt$ is surjective.
Then the localization map $\loc_S$ of \eqref{e:loc-S} is surjective.
\end{corollary}

\begin{proof}
Indeed, then
\[\im\Sigma_S\subseteq M_\Gt=\im\lambda_{v_0}\subseteq \im\Sigma_\SC\hs,\]
and we conclude by Corollary \ref{c:main}.
\end{proof}

\begin{proposition}[{Borel and Harder \cite[Theorem 1.7]{Borel-Harder}}]
\label{p:semisimple}
Let $G$ be a {\emm semisimple} group over a number field $F$,
and let $S\subset \Vm(F)$ be a subset
such that the complement $\SC$ of $S$ contains a {\emm finite} place $v_0\in\Vm_f(F)$.
Then the localization map $\loc_S$ of \eqref{e:loc-S} is surjective.
\end{proposition}

\begin{proof}
Since $G$ is semisimple, the $\G$-module $M$ is finite,
and so are the groups $M_\G$ and $M_\Gw$
where $w$ is a place of $E$ over $v_0$\hs.
It follows that
\[M_\Gwt=M_\Gw\quad\ \text{and}\quad\ M_\Gt=M_\G\hs.\]
The natural homomorphism $M_\Gw\to M_\G$ is clearly surjective.
Therefore, the  homomorphism
\begin{equation*}
\omega_{v_0}\colon\, M_\Gwt=M_\Gw\,\lra\, M_\G= M_\Gt
\end{equation*}
is surjective.
Since $v_0$ is finite, we have $\im\lambda_{v_0}=\im\omega_{v_0}$,
whence the map $\lambda_{v_0}$ is surjective.
We conclude by Corollary \ref{c:v0-sur}.
\end{proof}

\section{Exact sequence}
\label{s:exact}
In this section we construct an exact sequence that we shall use in Section \ref{s:PR}.

\begin{theorem}\label{t:exact}
A finite group $\G$ and a short exact sequence of $\G$-modules
\begin{equation}\label{e:1-2-3}
0\to B_1\labelto i B_2\labelto j B_3\to 0
\end{equation}
give rises to an exact sequence
\begin{multline}\label{e:1-2-3-long}
(B_1)_\Gt \labelto{i_*} (B_2)_\Gt \labelto{j_*} (B_3)_\Gt\labelto\delta\\
\Q/\Z\otimes_\Z (B_1)_\G \labelto{i_*} \Q/\Z\otimes_\Z(B_2)_\G
     \labelto{j_*} \Q/\Z\otimes_\Z(B_3)_\G\to 0
\end{multline}
depending functorially on $\G$ and on the sequence \eqref{e:1-2-3}.
\end{theorem}

\begin{subsecc}\label{ss-delta}
We specify the homomorphism $\delta$.
Let $x_3\in B_3$ be such that the image $(x_3)_\G$ of $x_3$
in $(B_3)_\G$ is contained in  $(B_3)_\Gt$\hs.
This means that there exist $n\in\Z_{>0}$ and $y_{3,\g}\in B_3$
such that
$$nx_3=\sum_{\g\in\G}\big(\upgam y_{3,\g}-y_{3,\g}\big).$$
We lift $x_3$ to some $x_2\in B_2$\hs, we lift each $y_{3,\g}$ to some $y_{2,\g}\in B_2$\hs,
and we consider the element
\begin{equation*}
z_2=nx_2-\sum_{\g\in\G}\big(\upgam y_{2,\g}-y_{2,\g}\big).
\end{equation*}
Then $j(z_2)=0\in B_3$\hs, whence $z_2=i(z_1)$ for some  $z_1\in B_1$\hs.
We consider the image $(z_1)_\Gtf$ of $z_1\in B_1$ in $(B_1)_\Gtf$\hs, and we put
\[\delta\big(\hs(x_3)_\G\big)=  \frac1n\otimes (z_1)_\Gtf \hs
    \in \hs\Q/\Z\otimes_\Z (B_1)_\Gtf=\Q/\Z\otimes_\Z (B_1)_\G\]
where we write $\frac1n$ for the image in $\Q/\Z$ of $\frac1n\in \Q$.

Below we give the proof of Theorem \ref{t:exact} suggested by Vladimir Hinich (private communication).
For another proof, due to Alexander Petrov, see \cite{SashaP}.
\end{subsecc}

\begin{subsecc}{\it Proof Theorem \ref{t:exact}  due to Vladimir Hinich.}\ 
The functor  from  the category $\G$-modules to the category of abelian groups
\[B\rsa \Q/\Z\otimes_\Z B_\G\]
is the same as
\[ B\rsa\Q/\Z\otimes_\Lam\! B\]
where $\Lam=\Z[\G]$ is the group ring of $\G$.
From the short exact sequence of $\G$-modules  \eqref{e:1-2-3},
we obtain a long exact sequence
\begin{multline*}
\dots\to \Tor_1^\Lam(\Q/\Z,B_1) \labelto{i_*} \Tor_1^\Lam(\Q/\Z,B_2) \labelto{j_*} \Tor_1^\Lam(\Q/\Z,B_3)\labelto\delta\\
\Q/\Z\otimes_\Lam\! B_1\labelto{i_*} \Q/\Z\otimes_\Lam\! B_2\labelto{j_*} \Q/\Z\otimes_\Lam\! B_3\to 0
\end{multline*}
depending functorially on $\G$ and on  \eqref{e:1-2-3}; see Weibel \cite{Weibel}.
Now Theorem \ref{t:exact} follows from the next proposition.
\qed
\end{subsecc}

\begin{proposition}\label{p:Q/Z-B-BGt}
For a finite group $\G$ and a $\G$-module $B$, there is a canonical and functorial isomorphism
\[\Tor_1^\Lam (\Q/\Z,B)\isoto B_\Gt\hs\]
where $\Lam=\Z[\G]$.
\end{proposition}

\begin{proof}
Consider the short exact sequence
\[0\to \Z\to \Q\to\Q/\Z\to 0\]
regarded as a short exact sequence of $\G$-modules with trivial action of $\G$.
Tensoring with $B$, we obtain a long exact sequence
\begin{equation} \label{e:long-tensor}
\dots\to\Tor_1^\Lam(\Q,B)\to \Tor_1^\Lam(\Q/\Z,B)\to \Z\otimes_\Lam\! B
      \to \Q\otimes_\Lam\! B \to \Q/\Z\otimes_\Lam\! B\to 0.
\end{equation}
We have canonical isomorphisms
\[\Z\otimes_\Lam\! B=B_\G\quad\ \text{and}\quad\
   \ker\big[\Z\otimes_\Lam\! B\to \Q\otimes_\Lam\! B\big]= B_\Gt\hs.\]
By Lemma \ref{l:Tor-Q} below, we have  $\Tor_1^\Lam(\Q,B)=0$,
and the proposition follows from   \eqref{e:long-tensor}.
\end{proof}

\begin{lemma}\label{l:Tor-Q}
For a finite group $\G$ and any $\G$-module $B$, we have
\[\Tor_1^\Lam (\Q,B)=0\]
where $\Lam=\Z[\G]$.
\end{lemma}

\begin{proof}
Let
\[P_\bullet:\quad\dots\to P_2\to P_1\to P_0\to\Z\to 0\]
be a $\Lam$-free resolution of the trivial $\G$-module $\Z$,
for example, the standard complex; see Atiyah and Wall \cite[Section 2]{AW}.
Tensoring with $\Q$ over $\Z$, we obtain a flat resolution of $\Q$
\[\dots\to \Q\otimes_\Z P_2\to \Q\otimes_\Z P_1\to \Q\otimes_\Z P_0\to\Q\to 0.\]
Tensoring with $B$ over $\Lam=\Z[\G]$, we obtain the complex $(\Q\otimes_\Z P_\bullet)\otimes_\Lam B$\hs:
\begin{equation}\label{e:QPB}
\quad\dots\to (\Q\otimes_\Z P_2)\otimes_\Lam B\to (\Q\otimes_\Z P_1)\otimes_\Lam B
    \to (\Q\otimes_\Z P_0)\otimes_\Lam B\to\Q\otimes_\Lam B\to 0.
\end{equation}
By definition, $\Tor_1^\Lam(\Q,B)$ is the first homology group of this complex.

However, we can obtain the complex \eqref{e:QPB} from $P_\bullet$
by tensoring first with $B$ over $\Lam$,
and after that with $\Q$ over $\Z$:
\[\Q\otimes_\Z\big(P_\bullet\otimes _\Lam B\big)\hs\cong\hs (\Q\otimes_\Z P_\bullet)\otimes_\Lam B.\]
Since $\Q$ is a flat $\Z$-module, we obtain canonical isomorphisms
\[\Tor_1^\Lam(\Q,B)\cong \Q\otimes_\Z \Tor_1^\Lam(\Z,B)=\Q\otimes_\Z H_1(\G,B).\]

Now, since the group $\G$ is finite, the abelian group $H_1(\G,B)$
is killed by multiplication by $\#\G$;
see, for instance, Atiyah and Wall \cite[Section 6, Corollary 1 of Proposition 8]{AW}.
It follows that $\Q\otimes_\Z H_1(\G,B)=0$.
Thus $\Tor_1^\Lam (\Q,B)=0$, which completes the proofs
of Lemma \ref{l:Tor-Q}, Proposition \ref{p:Q/Z-B-BGt}, and Theorem \ref{t:exact}.
\end{proof}

Alternatively, one can check directly that the map $\delta$ constructed in Subsection \ref{ss-delta}
is well-defined (does not depend on the choices made) and that the sequence \eqref{e:1-2-3-long} is exact.

\section{Surjectivity for a reductive group with nice radical}
\label{s:PR}

In this section we prove the following theorem that gives a sufficient condition
for the surjectivity of the localization map
\eqref{e:loc-S} for a reductive $F$-group $G$
in terms of the  {radical (largest central torus) of $G$.

\begin{theorem}\label{t:PR-generalized}
Let $G$ be a reductive group over a number field $F$,
and let $C$ denote the radical of $G$.
Write $\ov G=G/C$, which is a semisimple group,
and consider the short exact sequence of fundamental groups \cite[Lemma 1.5]{Borovoi-Memoir}
\[ 0\to M_C\to M\to \ov M\to 0\]
 where
 $$M_C=\pi_1 (C)=\X_*(C),\quad  M=\pi_1 (G),\quad  \ov M=\pi_1 (\ov G).$$
We define $\Gamma=\Gal(E/F)$ for $M$ as in Subsection \ref{ss:local-results}.
Let $S\subset \Vm(F)$ be a subset, and assume
that $\SC$ contains a {\emm finite} place $v_0$ such that
\begin{equation}\label{e:Z-Gw-G}
\im\hs [\Gw\to \Aut M_C] = \im\hs [\G\to \Aut M_C]
\end{equation}
where  $w$ is a place of $E$ over $v_0$\hs.
Then the localization map  $\loc_S$ of \eqref{e:loc-S} is surjective.
\end{theorem}

\begin{proof}
It follows from \eqref{e:Z-Gw-G} that $(M_C)_\Gw=(M_C)_\G$\hs,
whence
$$(M_C)_\Gwt=(M_C)_\Gt\quad\ \text{and}\quad\ \Q/\Z\otimes_\Z (M_C)_\Gw=\Q/\Z\otimes_\Z(M_C)_\G\hs.$$
Using Theorem \ref{t:exact}, we construct an exact commutative diagram
\[
\xymatrix{
(M_C)_\Gwt \ar[r]\ar@{=}[d] &M_\Gwt\ar[r]\ar[d]^-\omega  &\ov M_\Gwt\ar[r]\ar[d]^-{\ov\omega} &\Q/\Z\otimes_\Z\!(M_C)_\Gw \ar@{=}[d] \\
(M_C)_\Gt \ar[r]            &M_\Gt\ar[r]                 &\ov M_\Gt\ar[r]              &\Q/\Z\otimes_\Z\!(M_C)_\G
}
\]
Since $\ov G$ is semisimple, its algebraic fundamental group $\ov M$ is finite, and therefore
the homomorphism $\ov\omega$ in the diagram above is surjective;
see the proof of Proposition \ref{p:semisimple}.
By a four lemma, the homomorphism
\[\omega=\omega_{v_0}\colon\hs M_\Gwt\to M_\Gt\]
is surjective as well.
Since $v_0$ is finite, the map
\[\alpha_{v_0}\colon\, H^1(F_{v_0},G)\to H^1_\ab(F_{v_0},G)\to  M_\Gwt\]
is bijective, and therefore the map
\[\lambda_{v_0}\colon\, H^1(F_{v_0},G)\to H^1_\ab(F_{v_0},G)\to M_\Gwt \labelt\omega M_\Gt\]
is surjective.
We conclude by Corollary \ref{c:v0-sur}.
\end{proof}

\begin{corollary}[{Prasad and Rapinchuk \cite[Proposition 2(a)]{Prasad-Rapinchuk}}\hs]
Let $G$ be a reductive group over a number field $F$,
and let $C$ denote the radical of $G$.
Assume that the $F$-torus $C$ is split and that  $\SC$ contains
a  finite place $v_0$.
Then the localization map $\loc_S$ of \eqref{e:loc-S} is surjective.
\end{corollary}

\begin{proof}
We define $E$, $\G$, and $\Gw$ for $M=\pi_1 (G)$ as in Subsection \ref{ss:local-results}.
Then $\im\hs [\G\to \Aut M_C]=\{1\}$, and hence \eqref{e:Z-Gw-G} holds.
We conclude by Theorem \ref{t:PR-generalized}.
\end{proof}

\begin{proof}[Proof of Corollary \ref{c:Prasad-Rapinchuk}]
We define $E$, $\G$, and $\Gw$ for $M=\pi_1 (G)$
as in Subsection \ref{ss:local-results}. We have
\[ \im \big[\Gw\to \Aut M_C\big]\hs\subseteq\hs\im\big[\G\to \Aut M_C\big],
     \quad\ \#\im\big[\Gw\to \Aut M_C\big]\ \big |\ p,
     \quad\   \im\big[\Gw\to \Aut M_C\big]\neq \{1\}. \]
It follows that  \eqref{e:Z-Gw-G} holds.
We conclude by Theorem \ref{t:PR-generalized}.
\end{proof}

\appendix

\newcommand{\GscR}{{\hs G^\ssc\!\labelt\rho G,\ }}
\newcommand{\GscF}{{\hs G^\ssc(\Fs)\labelt\rho G(\Fs),\ }}

\section{Abelianization}
\label{s:ab}

\begin{subsecc}\label{ss:G-Gss-Gsc}
Let $G$ be a reductive group over a field $F$ {\em of arbitrary characteristic}.
We consider the  homomorphism
$\rho\colon  G^\ssc\to G$ of Subsection \ref{ss:pi1}.

The group $G$ acts by conjugation on itself on the left, and by functoriality $G$ acts on $G^\ssc$.
We obtain an action
\[\theta\colon G\times G^\ssc\to G^\ssc,\quad\  (g,s)\mapsto {}^g\! s.\]
On $\Fbar$-points, if $s\in G^\ssc(\Fbar)$, $g_1\in G(\Fbar)$,
$g_1=\rho(s_1)\cdot  z_1$ with $s_1\in G^\ssc(\Fbar)$, $z_1\in Z_G(\Fbar)$,
then
\[\theta(g_1,s)={}^{g_1}\hm s= s_1 s s_1^{-1}.\]
Since the groups $G$ and $G^\ssc$ are smooth, this formula uniquely determines $\theta$.
The action $\theta$ has the following properties:
\begin{align*}
^{\rho(s)}\hm s'&=s\hs s' s^{-1},\\
\rho({}^{g_1}\hm s')&=g_1\hs \rho(s') g_1^{-1}
\end{align*}
for $g_1\in G(\Fbar),\ s,s'\in G^\ssc(\Fbar)$.
In other words, $(G^\ssc,G,\rho,\theta)$ is a (left) {\em crossed module of algebraic groups};
see for instance  \cite[Definition 3.2.1]{Borovoi-Memoir}.
We write it as $\big(\GscR \theta\big)$, and we regard it as a complex in degrees $-1,\, 0.$
On $\Fs$-points we obtain a  $\Gal(\Fs/F)$-equivariant crossed module  $\big(\GscF \theta\big)$
where $\Fs$ is the separable closure of $F$ in $\Fbar$.
\end{subsecc}

\begin{subsecc}
Deligne \cite[Section 2.0.2]{Deligne}
noticed that the commutator map
\[[\,\hs,\hs]\colon G\times G\to G,\quad\ g_1,g_2\mapsto [g_1,g_2]\coloneqq g_1 g_2 g_1^{-1} g_2^{-1}\]
lifts to a certain map (morphism of $F$-varieties)
\[\{\,,\}\colon G\times G\to G^\ssc,\quad\ g_1,g_2\mapsto \{g_1,g_2\}\]
as follows.
The commutator map
\[G^\ssc\times G^\ssc\to G^\ssc,\quad\ s_1,s_2\mapsto [s_1,s_2]\coloneqq s_1 s_2 s_1^{-1} s_2^{-1}\]
clearly factors via  a morphism of $F$-varieties
\[(G^\ssc)^\ad\times (G^\ssc)^\ad\to G^\ssc\]
where $(G^\ssc)^\ad=G^\ssc/Z_{G^\ssc}$ and $Z_{G^\ssc}$ denotes the center of $G^\ssc$.
Identifying $(G^\ssc)^\ad$ with $G^\ad\coloneqq G/Z_G$\hs, we obtain the desired morphism of $F$-varieties
\[\{\,,\}\colon G\times G\to G^\ad\times G^\ad\to G^\ssc.\]
On $\Fbar$-points, if $g_1,g_2\in G(\Fbar),\ g_1=\rho(s_1) z_1\hs,\ g_2=\rho(s_2) z_2$
where $s_1,s_2\in G^\ssc(\Fbar),\ z_1,z_2\in Z_G(\Fbar)$, then
\[ \{g_1,g_2\}=[s_1,s_2]=s_1s_2s_1^{-1}s_2^{-1}.\]
Since $G$ and $G^\ssc$ are smooth, this formula uniquely determines $\{\,,\}$.
The constructed map $\{\,,\}$ satisfies the following equalities of  Conduch\'e \cite[(3.11)]{Conduche}):
\begin{align*}
&\rho\big(\{g_1,g_2\}\big)=[g_1,g_2];\\
&\big\{\rho(s_1),\rho(s_2)\big\}=[s_1,s_2];\\
&\{g_1,g_2\}=\{g_2,g_1\}^{-1};\\
&\{g_1g_2,\hs g_3\}=\{g_1 g_2 g_1^{-1},\,g_1 g_3 g_1^{-1}\}\hs\{g_1,g_3\}.
\end{align*}
In other words, the map $\{\,,\}$ is a {\em symmetric braiding} of the crossed module $(G^\ssc,G, \rho,\theta)$.
We denote by $G_\ab$ the corresponding {\em stable} (=symmetrically braided) crossed module:
\[G_\ab=\big(\GscR \theta,\{,\hmm\}\hs\big).\]

Let $\varphi\colon G\to H$ be a homomorphism of reductive $F$-groups.
It induces a homomorphism $\varphi^\ssc\colon G^\ssc\to H^\ssc$.
It is easy to see that
\[ {}^{\varphi(g)}\varphi^\ssc(s)=\varphi^\ssc(\hs^g\hm s)\quad\ \text{for all }\ \,g\in G(\Fbar),\, s\in G^\ssc(\Fbar).\]
Thus we obtain a morphism of crossed modules
\[ (G^\ssc\to G,\hs \theta_G)\,\to\,(H^\ssc\to H,\hs \theta_H)\]
with obvious notations. Moreover, we have
\[\big\{\varphi(g_1),\varphi(g_2)\big\}_H= \varphi^\ssc\big(\{g_1,g_2\}_G\big) \quad\ \text{for all}\ \,g_1,g_2\in G(\Fbar)\]
with obvious notations; see \cite{Borovoi-MO} for a proof.
Thus we obtain a morphism of stable crossed modules
\begin{equation}\label{e:varphi-stable}
 \big(G^\ssc\!\to G,\,\theta_G,\{,\hmm\}_G\hs\big)\,\lra\, \big(H^\ssc\!\to H,\, \theta_H,\{,\hmm\}_H\hs\big).
\end{equation}
\end{subsecc}

\begin{subsecc}\label{ss:ab}
In this appendix, we denote by $H^1$ and $\H^1$
the first {\em Galois} cohomology and hypercohomology.
One can define the first Galois (hyper)cohomology of the $\Gal(\Fs/F)$-equivariant crossed module
\begin{equation}\label{e:BN}
\H^1(F,\hs \GscR\theta)\coloneqq
   \H^1(\Gal(\Fs/F),\GscF \theta\hs\big);
\end{equation}
see \cite[Section 3]{Borovoi-Memoir} or Noohi \cite[Section 4]{Noohi}.
A priori it is just a pointed set. However, using the symmetric braiding $\{\,,\}$,
one can define a structure of abelian group on the pointed set \eqref{e:BN};
see Noohi \cite[Corollaries 4.2 and 4.5]{Noohi}.
We denote the obtained abelian group by
\[H^1_\ab(F,G)=\H^1(F, G_\ab)\coloneqq \H^1\big(\hs\Gal(\Fs/F),\hs \GscF \theta,\hs\{,\hmm\}\hs\big).\]
A homomorphism of reductive $F$-groups $\varphi\colon G\to H$
induces a morphism of stable crossed modules
\eqref{e:varphi-stable}, which in turn induces a homomorphism of abelian groups
\[\varphi_\ab\colon H^1_\ab(F,G)\to H^1_\ab(F, H).\]
Thus $G\rightsquigarrow H^1_\ab(F,G)$ is a functor
from the category of reductive $F$-group to the category of abelian groups.
\end{subsecc}

\begin{subsecc}
The morphism of crossed modules (but not of stable crossed modules)
\[i_G\colon (1\to G)\,\into\, (G^\ssc\to G)\]
induces a morphism of pointed sets
\[(i_G)_*\colon \H^1(F, 1\to G)\to \H^1(F,\hs G^\ssc\!\labelt\rho G).\]
The {\em abelianization map}
is the composite morphism of pointed sets
\[\ab\colon\, H^1(F,G)=\H^1(F, 1\to G)\,\labeltoo{(i_G)_*}\,
   \H^1\big(F, \GscR\theta\big)=\H^1(F,G_\ab)\eqqcolon H^1_\ab(F,G).\]
Here $\H^1\big(F, \GscR\theta\big)$ and $\H^1(F,G_\ab)$ are the same sets,
but $\H^1(F,G_\ab)$ is endowed with the structure of abelian group
coming from the symmetric braiding $\{\,,\}$.
\end{subsecc}

\begin{subsecc}
For a maximal torus  $T\subseteq G$, we consider the homomorphism
\[ \rho\colon T^\ssc\to T\]
of Subsection \ref{ss:pi1},
which we  regard as a stable crossed module
with the trivial action $\theta_T$ of $T$ on $T^\ssc$
and the trivial symmetric braiding \,$\{,\hmm\}_T\colon T\times T\to T^\ssc$.
We may and shall identify the first Galois hypercohomology
of this stable crossed module
with the usual first Galois hypercohomology of the complex $T^\ssc\!\labelt\rho T$ in degrees $-1,\,0$:
\[ \H^1\big(F,\, T^\ssc\!\labelt\rho T,\,\theta_T, \{,\hmm\}_T\big)=\H^1(F,T^\ssc\!\labelt\rho T).\]
The morphism of stable crossed modules
\begin{equation}\label{e:iota}
j_T\colon\, \big(T^\ssc\!\labelt\rho T,\,\theta_T, \{,\hmm\}_T\big)\,\into\, \big(\hs G^\ssc\!\labelt\rho G,\,\theta, \{,\hmm\}\hs\big)
\end{equation}
is an {\em equivalence} (quasi-isomorphism),
that is, it induces isomorphisms of $F$-group schemes
\[\ker[T^\ssc\!\to T]\isoto \ker[G^\ssc\!\to G]\quad\ \text{and}
    \quad\ \coker[T^\ssc\!\to T]\isoto \coker[G^\ssc\!\to G].\]
Following an idea sketched by Labesse and Lemaire \cite{LL},
we observe that  \eqref{e:iota}   induces iso\-morphisms on groups of $\Fs$-points
\begin{align*}
\ker\big[T^\ssc(\Fs)\to T(\Fs)\big]&\isoto \ker\big[G^\ssc(\Fs)\to G(\Fs)\big]\\
\coker\big[T^\ssc(\Fs)\to T(\Fs)\big]&\isoto \coker\big[G^\ssc(\Fs)\to G(\Fs)\big]
\end{align*}
(in arbitrary characteristic);
see Theorem \ref{t:Zev} in Appendix \ref{s:Zev} below.
It follows that the induced map on Galois gypercohomology
\begin{equation*}
(j_T)_*\colon\, \H^1(F,T^\ssc\!\to T)\,\lra\,
    \H^1\big(F,\GscR \theta, \{,\hmm\}\hs\big)\eqqcolon H^1_\ab(F,G)
\end{equation*}
is an isomorphism of abelian groups; see Noohi \cite[Proposition 5.6]{Noohi}.
This shows that the abelian group structure
on the pointed set $\H^1\big(F, \GscR\theta\big)$
defined using the bijection $(j_T)_*$ (as in \cite[Section 3.8]{Borovoi-Memoir}\hs) coincides with
the abelian group structure defined by the symmetric braiding $\{\,,\}$.
\end{subsecc}

\def\fppf{{\rm fppf}}

\begin{remark}
Gonz\'alez-Avil\'es \cite{GA} defined the abelian fppf cohomology group
$H^1_{\fppf,\hs\ab}(X,G)$ and the abelianization map
 \[\ab\colon H^1_\fppf(X,G)\to H^1_{\fppf,\hs\ab}(X,G)\]
for a reductive group scheme $G$ over an arbitrary base scheme $X$,
which includes the case of a reductive group
over a field $F$ of arbitrary characteristic.
However, his definition uses the center $Z_G$ of $G$,
and hence it is functorial only with respect to the {\em normal} homomorphisms
$G_1\to G_2$ (homomorphisms with normal image, hence sending  $Z_{G_1}$  to $Z_{G_2}$)\hs),
whereas our definition above (over a field only) is functorial with respect to all homomorphisms.
\end{remark}

\def\cok{{\rm cok}}

\section{Equivalence on $\Fs$-points in arbitrary characteristic}
\label{s:Zev}

\centerline{\em Zev Rosengarten}
\medskip

In this appendix we prove the following theorem:

\begin{theorem}\label{t:Zev}
Let $F$ be a field of arbitrary characteristic and let $\Fs$ be a fixed separable closure of $F$.
Let \[\rho\colon G^\ssc\onto [G,G]\into G\]
be as in Subsection  \ref{ss:pi1}.
Let $T\subseteq G$ be a maximal torus.
We write $T^\ssc=\rho^{-1}(T)$.
Then the morphism of crossed modules
\[\big(T^\ssc(\Fs)\to T(\Fs)\big)\,\lra\,  \big(G^\ssc(\Fs)\to G(\Fs)\big)\]
is an equivalence (quasi-isomorphism).
\end{theorem}

\begin{proof}
We must show that the maps
\begin{equation}\label{e:ker}
 i_{\ker}\colon \ker\big[T^\ssc(\Fs) \to T(\Fs)\big] \,\lra\, \ker\big[G^\ssc(\Fs) \to G(\Fs)\big]
\end{equation}
and
\begin{equation}\label{e:coker}
i_{\cok}\colon \coker\big[T^\ssc(\Fs) \to T(\Fs)\big] \,\lra\, \coker\big[G^\ssc(\Fs) \to G(\Fs)\big]
\end{equation}
are isomorphisms.

For \eqref{e:ker},  the injectivity is obvious.
Moreover, any element of $\ker\big[G^\ssc(\Fs) \to G(\Fs)\big]$
lies in the preimage $T^\ssc$ of $T$, hence it is an element of $T^\ssc(\Fs)$
and of $\ker\big[T^\ssc(\Fs) \to T(\Fs)\big]$,
which gives the surjectivity of $i_{\ker}$.

We prove the injectivity of \eqref{e:coker}.
Let   $[t]\in \coker\big[T^\ssc(\Fs) \to T(\Fs)\big]$, $t\in T(\Fs)$,
and $[t]\in\ker i_\cok$\hs; then $t=\rho(s)$ for some $s\in G^\ssc(\Fs)$.
Since $T^\ssc=\rho^{-1}(T)$, we see that $s\in T^\ssc(\Fs)$, whence $[t]=1$, as required.

We prove the surjectivity of \eqref{e:coker}.
Let $C \subseteq G$ denote the radical (largest central torus) of $G$.
Then the  map
\[\psi\colon C \times G^\ssc \to G,\qquad (c,s)\mapsto c\cdot \rho(s)\ \,\text{for}\ c\in C,\, s\in G^\ssc\]
is surjective with central kernel $Z\cong \rho^{-1}(C\cap [G,G])$ (which might be non-smooth).
We have an exact  commutative diagram of $F$-group schemes
\[
\xymatrix@R=6mm{
1 \ar[r] &Z\ar[r]\ar@{=}[d]  &C \times T^\ssc\ar[r]^-{\psi_T}\ar[d]  &T\ar[r]\ar[d]   &1 \\
1 \ar[r] &Z\ar[r]            &C \times G^\ssc\ar[r]^-\psi               &G \ar[r]        &1
}
\]
in which the maps on $\Fs$-points
\[\psi_T\colon C(\Fs)\times T^\ssc(\Fs)\to T(\Fs)\quad\ \text{and}
\quad\ \psi\colon C(\Fs)\times G^\ssc(\Fs)\to G(\Fs)\]
might not be surjective.
This diagram gives rise to an exact commutative diagram of fppf cohomology groups
\[
\xymatrix@R=6mm{
C(\Fs) \times T^\ssc(\Fs)\ar[r]^-{\psi_T}\ar[d]  &T(\Fs)\ar[r]\ar[d]   &H_\fppf^1(\Fs, Z)\ar[r]\ar@{=}[d]   &H_\fppf^1(\Fs,C \times T^\ssc)=1\ar[d]\\
C(\Fs) \times G^\ssc(\Fs)\ar[r]^-\psi                &G(\Fs)\ar[r]     &H_\fppf^1(\Fs, Z)\ar[r]             &H_\fppf^1(\Fs,C \times G^\ssc)=1
}
\]
in which the rightmost term in both rows is trivial because $\Fs$ is separably closed
and the $F$-groups $C \times T^\ssc$,  \,$C \times G^\ssc$ are smooth.
The latter diagram shows that
\[G(\Fs) = T(\Fs)\cdot \psi\big(\hs C(\Fs) \times G^\ssc(\Fs)\hs\big)=
    T(\Fs)\cdot C(\Fs)\cdot\rho\big(G^\ssc(\Fs)\big) = T(\Fs)\cdot \rho\big(G^\ssc(\Fs)\big),\]
whence the surjectivity of \eqref{e:coker}.
\end{proof}

\paragraph*{\sc Acknowledgements.}
The author is very grateful to Vladimir Hinich
for his short proof of Theorem \ref{t:exact}
(together with  Proposition \ref{p:Q/Z-B-BGt} and Lemma \ref{l:Tor-Q})
included in this paper instead of the original lengthy proof of the author,
and to Zev Rosengarten for his Appendix \ref{s:Zev}.
Both the author of the article and the author of the appendix
thank the anonymous referee for his/her insightful comments and suggestions.


\begin{thebibliography}{CGP15}

\bibitem
{AW}
M.F. Atiyah and C.T.C. Wall,
\newblock {\em Cohomology of groups},
\newblock In Algebraic {N}umber {T}heory ({P}roc. {I}nstructional {C}onf.,
  {B}righton, 1965),  94--115. Thompson, Washington, D.C., 1967.

\bibitem
{Borel-Harder}
A. Borel and G. Harder,
{\em Existence of discrete cocompact subgroups of reductive groups over local fields},
J. Reine Angew. Math. 298 (1978), 53–64.

\bibitem
{Borel-Tits-C}
A. Borel et J. Tits,
\newblock {\em Compl\'{e}ments \`a l'article: ``{G}roupes r\'{e}ductifs''},
\newblock  Inst. Hautes \'{E}tudes Sci. Publ. Math. { 41} (1972), 253--276.

\bibitem
{Borovoi-Memoir}
M. Borovoi,
{\em Abelian Galois cohomology of reductive groups},
Mem. Amer. Math. Soc. {132} (1998), no. 626.

\bibitem%
{Borovoi-MO}
M. Borovoi (\url{https://mathoverflow.net/users/4149/mikhail-borovoi}),
{\em Is Deligne's braiding functorial?}, URL (version: 2022-12-06):
\url{https://mathoverflow.net/q/436002}.

\bibitem
{BT-arXiv}
M. Borovoi and  D.A. Timashev,
{\em Galois cohomology and component group of a real reductive group},
To appear in Israel J. Math.,
\url{https://doi.org/10.48550/arXiv.2110.13062}.

\bibitem%
{Breen}
L. Breen,
Letter to Borovoi of May 27, 1991.

\bibitem
{CT-flasque}
J.-L. Colliot-Th\'el\`ene,
{\em R\'esolutions flasques des groupes lin\'eaires connexes},
J. Reine Angew. Math. 618 (2008), 77–133.

\bibitem{Conduche}
D. Conduch\'e,
{\em Modules crois\'es g\'en\'eralis\'es de longueur 2,}
Proceedings of the Luminy conference on algebraic K-theory (Luminy, 1983).
J. Pure Appl. Algebra 34 (1984), no. 2--3, 155--178.

\bibitem
{CGP}
B. Conrad, O. Gabber, and G. Prasad,
{\em Pseudo-reductive groups},
2nd edition. New Mathematical Monographs, 26.
Cambridge University Press, Cambridge, 2015.

\bibitem{Deligne}
P. Deligne,
{\em Vari\'et\'es de Shimura: interpr\'etation modulaire,
et techniques de construction de mod\`eles canoniques},
Automorphic forms, representations and L-functions
(Proc. Sympos. Pure Math., Oregon State Univ., Corvallis, Ore., 1977), Part 2, pp. 247--289,
Proc. Sympos. Pure Math., XXXIII, Amer. Math. Soc., Providence, R.I., 1979.

\bibitem
{GA}
C.D. Gonz\'alez-Avil\'es,
{\em Abelian class groups of reductive group schemes,}
Israel J. Math. 196 (2013), no. 1, 175--214.

\bibitem
{Kottwitz-86}
R.E. Kottwitz,
{\em Stable trace formula: elliptic singular terms},
Math. Ann. 275 (1986), no. 3, 365--399.

\bibitem
{Labesse}
J.-P. Labesse,
{\em Cohomologie, stabilisation et changement de base,}
Appendix A by L. Clozel and Labesse, and Appendix B by L. Breen.
Ast\'erisque No. 257 (1999).

\bibitem
{LL}
J.-P. Labesse and B. Lemaire,
{\em On abelianized cohomology for reductive groups,}
Appendix B to the paper by E. Lapid and Z. Mao
``A conjecture on Whittaker-Fourier coefficients of cusp forms'',
J. Number Theory 146 (2015), 448--505.

\bibitem
{Merkurjev}
A.S. Merkurjev,
{\em K-theory and algebraic groups}, European Congress of Mathematics,
 Vol. II (Budapest, 1996), 43--72, Progr. Math., 169, Birkhäuser, Basel, 1998.

\bibitem{Noohi}
B. Noohi,
{\em Group cohomology with coefficients in a crossed module,}
J. Inst. Math. Jussieu 10 (2011), no. 2, 359--404.

\bibitem
{SashaP}
A. Petrov (SashaP) \url{(https://mathoverflow.net/users/39304/sashap)},
{\em Is this exact sequence known?}, URL (version: 2022-09-04):
\url{https://mathoverflow.net/q/429772}.

\bibitem
{Prasad-Rapinchuk}
G. Prasad and  A.S. Rapinchuk,
{\em On the existence of isotropic forms of semi-simple algebraic groups
over number fields with prescribed local behavior},
Adv. Math. 207 (2006), no. 2, 646–660.

\bibitem
{Serre}
J.-P. Serre, {\em Galois cohomology},
Springer-Verlag, Berlin, 1997.

\bibitem
{Weibel}
C.A. Weibel,
{\em An introduction to homological algebra},
Cambridge Studies in Advanced Mathematics, 38, Cambridge University Press, Cambridge, 1994.

\end{thebibliography}
\end{document}